\documentclass[12pt, a4paper]{article}
\usepackage{amsfonts}
\usepackage{mathrsfs}
\usepackage{longtable}
\usepackage{latexsym}
\usepackage{xy}
\usepackage{amsfonts,amsmath,amssymb,amsthm}
\usepackage{color}

\xyoption{all}

\newcommand{\bcen}{\begin{center}}     \newcommand{\ecen}{\end{center}}
\newcommand{\bay}{\begin{array}}      \newcommand{\eay}{\end{array}}
\newcommand{\beq}{\begin{eqnarray*}}      \newcommand{\eeq}{\end{eqnarray*}}

\def\Hom{\mathrm{Hom}}
\def\Ker{\mathrm{Ker}}
\def\Im{\mathrm{Im}}
\def\Coker{\mathrm{Coker}}

\def\End{\mathrm{End}}
\def\mod{\mathrm{mod}}

\def\id{\mathrm{Id}}

\def\lim{\mathrm{lim}}

\def\Coker{\mathrm{Coker}}

\def\add{\mathrm{add}}

\def\brick{\mathrm{brick}}
\def\ebrick{\mathrm{ebrick}}
\def\mbrick{\mathrm{mbrick}}

\begin{document}
	
	\newtheorem{theorem}{Theorem}[section]
	\newtheorem{proposition}[theorem]{Proposition}
	\newtheorem{lemma}[theorem]{Lemma}
	\newtheorem{corollary}[theorem]{Corollary}
	\newtheorem{remark}[theorem]{Remark}
	\newtheorem{example}[theorem]{Example}
	\newtheorem{definition}[theorem]{Definition}
	\newtheorem{question}[theorem]{Question}
	\numberwithin{equation}{section}

	\title{\large\bf
		ICE-closed subcategories and epibricks over recollements}
	
	\author{\large Jinrui Yang and Yongyun Qin* }

	\date{\footnotesize School of Mathematics, Yunnan Normal University, \\ Kunming, Yunnan 650500, China. E-mail:
		qinyongyun2006@126.com
	}
	
	\maketitle
	
	\begin{abstract} Let $(  \mathcal{A^{'}},\mathcal{A},\mathcal{A^{''}},i^\ast,i_\ast,i_!,j_!,j^\ast,j_\ast)$ be a recollement of abelian categories.
We proved that every ICE-closed subcategory (resp. epibrick, monobrick) in $\mathcal{A^{'}}$ or $\mathcal{A^{''}}$
can be extended to an ICE-closed subcategories (resp. epibrick, monobrick) in $\mathcal{A}$,
and the assignment $\mathcal{C}\mapsto j^*(\mathcal{C})$ defines a bijection between certain ICE-closed subcategories in $\mathcal{A}$
and
those in $\mathcal{A}''$.
Moreover, the ICE-closed subcategory $\mathcal{C}$ of $\mathcal{A}$ containing $i_\ast(\mathcal{A^{'}})$ admits a new recollement relative to ICE-closed subcategories $\mathcal{A^{'}}$ and
	$j^\ast(\mathcal{C})$ which induced from the original recollement when $j_!{j^\ast(\mathcal{C})}\subset\mathcal{C}$.
\end{abstract}
	\medskip
	
	
	\medskip
	
	
	\bigskip
	
	\section{\large Introduction}
	
	\indent\indent
In representation theory of algebras, several classes of subcategories of abelian categories
 have been investigated.
 For example, the notion of torsion class was introduced by Dickson \cite{Dic66}
and was studied in relation with tilting or $\tau$-tilting theory \cite{AIR14, BB80, DIG19}.
Recall that torsion classes are subcategories closed under taking extensions and quotients.
On the other hand,
wide subcategories, which are exact abelian subcategories closed under extensions, are investigated in connection
with torsion classes \cite{AS24, AP22, MS17}, ring epimorphisms \cite{MS17} and stability conditions \cite{BKT14}.

Recently, ICE-closed subcategories of abelian categories were introduced
in \cite{E22}, which are subcategories closed under taking images,
cokernels and extensions. Both torsion classes and wide subcategories are
typical examples of ICE-closed subcategories. Therefore, the notion of
ICE-closed subcategory can be seen as a generation of these two classes.
Moreover, ICE-closed subcategories are studied in relation with rigid modules \cite{E22}, torsion
classes \cite{E23, ES22} and t-structure \cite{Ska23}. We refer to \cite{ES23, HNW24, Han24}
for more discussions on ICE-closed subcategories.
In \cite{E21}, Enomoto considered two other kinds of subcategories, named monobricks
and epibricks. Indeed, a monobrick (resp. epibrick) is a full subcategory whose objects have division endomorphism
rings and every non-zero map between these objects
is either zero or monomorphic (resp. epimorphic).
Monobricks
and epibricks are
generalization of semisimple modules or semibricks, and they provide an effective approach
to investigate torsion-free classes and wide subcategories \cite{E21}.

Recollements of abelian categories are exact sequences of abelian categories
where both inclusion and quotient
 have left and right adjoints. They have been introduced by Beilinson, Bernstein and
 Deligne in the context of triangulated categories \cite{BBD82},
and were studied in the situation of abelian categories by many authors \cite{FP04, GKP21, Psa14, PV14}.
It should be noted that recollements of abelian categories
appear more naturally than those of triangulated categories, because any idempotent element of a ring induces a
recollement situation between module categories. Moreover,
recollements of abelian categories
are very useful in gluing and in reduction, see for example \cite{AKL11, Chen12, GM22, LVY14, Psa14}.
In particular, Zhang investigated
how wide subcategories and semibricks transfer in the recollements of abelian categories \cite{Z17, Z19}.
Following this, we consider the similar
gluing and reduction results for ICE-closed subcategories, epibricks and monobricks.

	\medskip
	
	 {\bf Theorem 1.1.} (Theorem~\ref{3.4}, ~\ref{3.5}, ~\ref{thm-bij} and ~\ref{thm-recoll})
		{\it
	   Let $(  \mathcal{A^{'}},\mathcal{A},\mathcal{A^{''}},i^\ast,i_\ast,i^!,$
$j_!,j^\ast,j_\ast)$ be a recollement
of abelian categories.
		
{\rm (1)} If $\mathcal{C}$ is an ICE-closed subcategory of $\mathcal{A^{'}}$, then $i_\ast(\mathcal{C})$ is an ICE-closed subcategory
of $\mathcal{A}$.
		 	
{\rm (2)} If $j_{\ast}\  (resp. \ j_!)$ is exact and $\mathcal{C}$ is an ICE-closed subcategory of $\mathcal{A}''$, then
$j_{\ast}(\mathcal{C})\  (resp. \ j_!(\mathcal{C}))$ is an ICE-closed subcategory of $\mathcal{A}$.
		 	
{\rm (3)} The functor $j^*$ gives a bijection $$\xymatrix@!=4pc{\{\mathcal{C}\in {_{i_{\ast}(\mathcal{A^{'}})}\mathcal{A}_{ice}}
\ | \  j_!j^*\mathcal{C}\subset  \mathcal{C}\}\ar[rr]^-{1:1} & &\ar[ll]\mathcal{A''}_{ice}},  $$
where ${_{i_{\ast}(\mathcal{A^{'}})}\mathcal{A}_{ice}}$ is the set of ICE-closed subcategories
 of $\mathcal{A}$ containing $i_{\ast}(\mathcal{A^{'}})$ and $\mathcal{A''}_{ice}$ is the set of ICE-closed
	subcategories of $\mathcal{A''}$.

{\rm (4)}  If ${\mathcal{C}}\in {_{i_{\ast}(\mathcal{A^{'}})}\mathcal{A}_{ice}}$ and
$j_!{j^\ast(\mathcal{C})}\subset\mathcal{C}$, then there is a new recollement of ICE-closed subcategories $(  \mathcal{A^{'}},\mathcal{C},j^\ast(\mathcal{C}),\overline{i^\ast},\overline{i_\ast},\overline{i^!},\overline{j_!},\overline{j^\ast},\overline{j_\ast})$.
	}		
	\medskip

To glue
epibricks and monobricks via recollements, we need the intermediate extension functor $j_{!*}$, see Definition~\ref{def-int-func}.

		 {\bf Theorem 1.2.}(Theorem~\ref{4.3}, ~\ref{4.4} and ~\ref{4.5})
{\it  Let $(  \mathcal{A^{'}},\mathcal{A},\mathcal{A^{''}},i^\ast,i_\ast,$
$i^!,j_!,j^\ast,j_\ast)$ be a recollement of abelian categories.

{\rm (1)} Both $i_*$ and $j_{!*}$
preserve epibricks and monobricks.

{\rm (2)} $j_*$ preserves monobricks and $j_!$ preserves epibricks.

{\rm (3)} If $\mathcal{C^{'}}$ is an epibrick in $\mathcal{A^{'}}$ and $\mathcal{C^{''}}$ is an epibrick in $\mathcal{A^{''}}$, then $i_\ast(\mathcal{C^{'}})\cup {j_{!*}(\mathcal{C^{''}})}$ is an epibrick in $\mathcal{A}$. The same is true for monobrick.

{\rm (4)} Assume $\mathcal{C^{'}}$ is an epibrick in $\mathcal{A^{'}}$ and $\mathcal{C^{''}}$ is an epibrick in $\mathcal{A^{''}}$.
If $i^*$ is exact, then $i_\ast(\mathcal{C^{'}})\cup {j_!(\mathcal{C^{''}})}$ is an epibrick in $\mathcal{A}$;
If $i^!$ is exact, then $i_\ast(\mathcal{C^{'}})\cup {j_*(\mathcal{C^{''}})}$ is an epibrick in $\mathcal{A}$. The same is true for monobrick.
}
	
	\medskip

The idea of the proof of Theorem 1.1 is similar to that
in \cite{Z17}, but additional ingredients are needed.
In view of Theorem 1.1 and Theorem 1.2, we generalize the
main result in \cite{LG24}, where the reduction for ICE-closed subcategories and epibricks
were considered over one-point extension algebras.

This paper is organized as follows. In section \ref{Section-definitions and conventions}, we give some definitions and notations
which will be used later.
In section \ref{Recollements of ICE-closed subcategories}, we investigate how ICE-closed subcategories transfer in
recollements of abelian categories. In section \ref{4}, we will observe
how to glue epibricks and monobricks, and in section \ref{5}, we apply our results
triangular matrix algebras, one-point extension algebras and Morita rings.

	\section{\large Definitions and conventions}\label{Section-definitions and conventions}
	
	\indent\indent Throughout this paper, all subcategories are assumed to be full subcategories,
and all functors between abelian categories are assumed to be additive.
First of all, we will recall some basic definitions about
ICE-closed subcategory, monobrick
and epibrick. These subcategories
were first defined in modules categories \cite{E22, E21}
and then generalized to exact categories \cite{HNW24}. In this
paper, we only focus on abelian categories.

	\begin{definition}{}
{\rm Let $\mathcal{C}$ be a subcategory of an abelian category $\mathcal{A}$.

		(1) $\mathcal{C}$ is {\it closed under images (resp. kernels, cokernels)} if for every map $\xymatrix{M\ar[r]^f & N}$ with
		$M,N\in{\mathcal{C}}$, we have $\Im f\in{\mathcal{C}}$ (resp. $\Ker f\in{\mathcal{C}}, \Coker f\in{\mathcal{C}}$).
		
		(2) $\mathcal{C}$ is {\it closed under extensions} if for every short exact sequence in $\mathcal{A}$
		 $$\xymatrix{0\ar[r] & N_{1}\ar[r] &N_{2}\ar[r] & N_{3}\ar[r] &0}$$
		with $N_{1},N_{3}\in{\mathcal{C}}$, we have $N_{2}\in{\mathcal{C}}$.
			
		(3) $\mathcal{C}$ is {\it closed under quotients} if for every exact sequence in $\mathcal{A}$ $$\xymatrix{N_{1}\ar[r] & N_{2}\ar[r] &0}$$
		with $N_{1}\in{\mathcal{C}}$, we have $N_{2}\in{\mathcal{C}}$.
				
			(4) (See \cite{E22, HNW24}) $\mathcal{C}$ is an {\it ICE-closed subcategory} if it is closed under images, cokernels and extensions.
			
			(5) (See \cite{Hov01}) $\mathcal{C}$ is a {\it wide subcategory} if it is closed under kernels, cokernels and extensions.
			
			(6) (See \cite{Dic66}) $\mathcal{C}$ is a {\it torsion class} if it is closed under quotients and extensions.
			}
	\end{definition}	

It is clear that all wide subcategories and torsion classes are ICE-closed subcategories, see \cite[Corollary 2.3]{LG24} for example.
Besides, the empty subcategory, the zero subcategory and the entire category of
an abelian category are all ICE-closed subcategories.
We denote by $\mathcal{A}_{ice}$ (resp. $\mathcal{A}_{tor}$) the set of all ICE-closed
subcategories (resp. torsion classes) of $\mathcal{A}$, and $_{\mathcal{C}}{\mathcal{A}}_{ice}$
(resp. $_{\mathcal{C}}{\mathcal{A}}_{tor}$) the set of all ICE-closed subcategories (resp. torsion classes)
of $\mathcal{A}$ containing a given subcategory $\mathcal{C}$.
	
\begin{definition}
{\rm Let $A$ be an object of an abelian category $\mathcal{A}$.
		
(1) $A$ is a {\it brick} if $\End _{\mathcal{A}}(A)$ is a division ring. We denote by $\brick \mathcal{A}$ the set of isoclasses of bricks in $\mathcal{A}$.
		
		(2) (See \cite{E21}) A subset $S\subseteq{\brick{\mathcal{A}}}$ is called an {\it epibrick} if every morphism between elements of $S$ is either zero or epimorphic in $\mathcal{A}$. We denote by $\ebrick\mathcal{A}$ the set of epibricks in $\mathcal{A}$.
		
		(3) (See \cite{E21}) A subset $S\subseteq{\brick{\mathcal{A}}}$ is called a {\it monobrick} if every morphism between elements of $S$ is either zero or monomorphic in $\mathcal{A}$. We denote by  $\mbrick\mathcal{A}$ the set of monobricks in $\mathcal{A}$.}
		
	\end{definition}

Next, let's recall the definition of recollements of abelian categories, see for instance \cite{BBD82, FP04, Psa18}.
	\begin{definition} {\rm Let $\mathcal{A}, \mathcal{A^{'}},  \mathcal{A^{''}}$ be abelian categories. A recollement of $\mathcal{A}$ relative
			to $\mathcal{A^{'}}$ and $\mathcal{A^{''}}$
is given by the diagram
		$$\xymatrix@!=4pc{ \mathcal{A}^{'} \ar[r]|{i_{\ast}} & \mathcal{A} \ar@<-3ex>[l]_{i^{\ast}}
			\ar@<+3ex>[l]^{i^!} \ar[r]|{j^{\ast}} & \mathcal{A}^{''}
			\ar@<-3ex>[l]_{j_!} \ar@<+3ex>[l]^{j_{\ast}}}  $$
 henceforth denoted by $( \mathcal{A^{'}},\mathcal{A},\mathcal{A^{''}},i^\ast,i_\ast,i^!,j_!,j^\ast,j_\ast)$, satisfying the following conditions:
		
		(1) $(i^{\ast},i_{\ast}), (i_{\ast},i^!), (j_{!},j^{\ast})$ and $(j^{\ast},j_{\ast})$ are adjoint
		pairs;

		(2) $i_{\ast}$, $j_!$ and $j_{\ast}$ are fully faithful functors;

		(3) $\Ker j^{\ast}=\Im i_{\ast}$ .
		
		 }
	\end{definition}
	
In the following remark we isolate some easily
 established properties of a recollement situation which will be useful later.
		\begin{remark}\label{2.4}
		{\rm (1) Both $i_{\ast}$ and $j^{\ast}$ are exact,
$i^*$ and $j_!$ are right exact, and $i^!$ and $j_*$ are left exact.
		
		(2) There are four natural isomorphisms $i^{\ast}i_{\ast}\cong{\id _{\mathcal{A}'}},
i^{!}i_{\ast}\cong{\id  _{\mathcal{A}'}},j^{\ast}j_{!}\cong{\id  _{\mathcal{A}''}}$ and $j^{\ast}j_{\ast}\cong{\id  _{\mathcal{A}''}}$
given by the units and the counits of the adjoint
 pairs, respectively.
		
		(3) (See \cite[Proposition 2.6]{Psa14} or \cite[Lemma 3.1]{FZ17}
For all $A\in{\mathcal{A}}$, there exist $M^{'}$ and $N^{'}$ in $\mathcal{A^{'}}$ such that the
		units and counits induce the following two exact sequences
		$$0\rightarrow i_{\ast}(M^{'}) \rightarrow j_{!}j^{\ast}(A) \rightarrow A \rightarrow i_{\ast}i^{\ast}(A) \rightarrow0,$$
		$$0\rightarrow i_{\ast}i^{!}(A) \rightarrow A \rightarrow j_{\ast}j^{\ast}(A)  \rightarrow i_{\ast}(N^{'}) \rightarrow0.$$
		}
	\end{remark}

Associated to a recollement there is a seventh functor, called the intermediate extension
functor, see for instance \cite{BBD82, FP04, Kuhn94}. This functor plays an important
role in gluing simple modules and semibricks \cite{BBD82,Z19}.		
\begin{definition}\label{def-int-func}{\rm
Let $(  \mathcal{A^{'}},\mathcal{A},\mathcal{A^{''}},i^\ast,i_\ast,i^!,j_!,j^\ast,j_\ast)$ be a recollement of abelian categories.
The {\it intermediate extension
functor} $j_{!*}:\mathcal{A}''\rightarrow \mathcal{A}$
is given by $j_{!*}(M)=\Im (\xymatrix{j_!M\ar[r]^{\gamma _M} & j_*M})$, where
$M\in \mathcal{A}''$ and $\gamma _M$ is the map corresponding to the identity of the
natural isomorphisms
 $$\Hom _{\mathcal{A}}(j_!M,j_*M)
\cong \Hom _{\mathcal{A}''}(M,j^*j_*M)\cong \Hom _{\mathcal{A}''}(M,M).$$}
\end{definition}	

We need the following properties from \cite[Lemma 2.2]{CS17}, \cite[Proposition 4.3]{FP04}
and \cite[Proposition 4.6 ]{Kuhn94}.
\begin{lemma}\label{lem-int-func}

{\rm (1)} The functor $j_{!*}$ is fully faithful;

{\rm (2)} The functor $j_{!*}$ preserves monomorphisms and epimorphisms;

{\rm (3)} $i^*j_{!*}=0$ and
$i^!j_{!*}=0$.

\end{lemma}

	\section{ICE-closed subcategories and recollements }\label{Recollements of ICE-closed subcategories}
	
	\indent\indent In this section, we will investigate how ICE-closed subcategories transfer in
recollements of abelian categories. The following lemma is well-known. Here, we give a proof for readers' convenience.
	
	\begin{lemma}\label{3.3}
			Let $\xymatrix{\mathcal{A}\ar[r]^F & \mathcal{B}}$ be a functor between abelian categories, and
$f: M\rightarrow N$ be a map in $\mathcal{A}$.
			
{\rm (1)} If $F$ is left exact, then $F(\ker f)\cong{\ker F(f)}$;
	
{\rm (2)} If $F$ is right exact, then $F(\Coker f)\cong{\Coker F(f)}$;
	
{\rm (3)} If $F$ is exact, then $F(\Im f)\cong{\Im F(f)}$.
	\end{lemma}	
	\begin{proof}
(1) Applying $F$ to the exact sequence $\xymatrix{0\ar[r] & Kerf\ar[r] & M\ar[r]^f & N}$,
we get an exact sequence $\xymatrix{0\ar[r] & F(\Ker f)\ar[r] & F(M)\ar[r]^{F(f)} & F(N)}$,
and then $F(\Ker f)\cong{\Ker F(f)}$ by the uniqueness of kernel.

(2) This can be proved with an argument entirely dual to (1).

(3) Since $\Im f=\Ker (\Coker f)$ and $\Im F(f)=\Ker (\Coker F(f))$, the statement follows from (1) and (2).
	\end{proof}
	
	\begin{lemma}\label{ho1}
Let $\xymatrix{\mathcal{A}\ar[r]^F & \mathcal{B}}$ be a fully faithful exact functor between abelian categories.
Assume $\mathcal{C}$ is a subcategory of $\mathcal{A}$
which is closed under images and cokernels. Then $F(\mathcal{C})$ is closed under images and cokernels.
	\end{lemma}
	\begin{proof}
Let $\xymatrix{M\ar[r]^f & N}$ be a map of $F(\mathcal{C})$.
Since $M,N\in F(\mathcal{C})$
and $F$ is fully faithful, we have that $M=F(M^{'}),N=F(N^{'})$ and $f=F(f^{'})$, where
$M^{'},N{'}\in{\mathcal{C}}$ and $f'\in{\Hom _{\mathcal{A}}(M^{'},N^{'})}$.
Since $\mathcal{C}$ is closed under images and cokernels, we infer that $\Im {f}^{'}\in{\mathcal{C}}$ and $\Coker {f}^{'}\in{\mathcal{C}}$.
Now it follows from Lemma~\ref{3.3} that $\Im f=\Im F(f^{'})\cong{F(\Im f^{'})}\in{F(\mathcal{C})}$,
and $\Coker f=\Coker F(f^{'})=F(\Coker f^{'})\in{F(\mathcal{C})}$.
Therefore, $F(\mathcal{C})$ is closed under images and cokernels.
	\end{proof}

From now on, let $(  \mathcal{A^{'}},\mathcal{A},\mathcal{A^{''}},i^\ast,i_\ast,i^!,j_!,j^\ast,j_\ast)$ be a recollement of abelian categories.
We are in a position to prove Theorem 1.1.
\begin{theorem}\label{3.4}
If $\mathcal{C}\in{\mathcal{A^{'}}_{ice}}$, then $i_{\ast}(\mathcal{C})\in{\mathcal{A}_{ice}}$.
	\end{theorem}
	\begin{proof}
Since $i_{\ast}$ is fully faithful and exact,
it follows from Lemma~\ref{ho1} that $i_{\ast}(\mathcal{C})$ is closed under images and cokernels.
Now we claim that
$i_{\ast}(\mathcal{C})$ is closed under extensions. Let
		$$\xymatrix@C=2pc{0\ar[r] & M\ar[r]^f & N\ar[r]^g & P\ar[r]&0}\eqno {\rm (3.1)}$$ be an
		 exact sequence in $\mathcal{A}$, where $M=i_{\ast}(M^{'}),P=i_{\ast}(P^{'}),M^{'}\in{\mathcal{C}}$ and $P^{'}\in{\mathcal{C}}$.
Applying the exact functor $j^{\ast}$, we have the following exact sequence:
		 $$\xymatrix@C=2.3pc{0\ar[r] & j^{\ast}i_{\ast}(M^{'})\ar[r]^{j^{\ast}(f)} &j^{\ast} (N)\ar[r]^{j^{\ast}(g)} & j^{\ast}i_{\ast}(P')\ar[r]&0}.$$ Since $\Ker j^{\ast}=\Im i_{\ast}$, we get that $j^{\ast}i_{\ast}(M^{'})=0=j^{\ast}i_{\ast}(P^{'})$ and then $j^{\ast}(N)=0$.
Thus $N\in \Ker j^{\ast}$, and then $N=i_{\ast}(N^{'})$
for some $N^{'}\in{\mathcal{A^{'}}}$.
Since $i_{\ast}$ is fully faithful, we have that $f=i_{\ast}(f^{'})$ and $g=i_{\ast}(g^{'})$,
where $f'\in{\Hom _{\mathcal{A}'}(M^{'},N^{'})}$ and $g'\in{\Hom _{\mathcal{A}'}(N^{'},P^{'})}$. Therefore, the exact sequence (3.1)
is of the form $$\xymatrix@C=2pc{0\ar[r] & i_*M'\ar[r]^{i_{\ast}(f^{'})} & i_*N'\ar[r]^{i_{\ast}(g^{'})} & i_*P'\ar[r]&0}. \eqno {\rm (3.2)}$$
Applying the left exact functor $i^!$ and using $i^!i_*\cong \id  _{\mathcal{A}'}$, we obtain the following exact sequence
$$\xymatrix@C=2pc{0\ar[r] & M'\ar[r]^{f'} & N'\ar[r]^{g'} & P'}.$$
Since $i^*i_*\cong\id _{\mathcal{A}'}$, we have that $\Coker g'\cong\Coker (i^*i_*(g'))$, which is isomorphic to $i^* \Coker (i_*(g'))$
by Lemma~\ref{3.3}. Moreover, the exact sequence (3.2) implies that $i_*(g')$ is an epimorphism,
and then $\Coker (i_*(g'))=0$. Therefore, we get that $\Coker g'\cong i^* (\Coker (i_*g')) =0$, which shows that
$g'$ is an epimorphism, and then the sequence $$\xymatrix@C=2pc{0\ar[r] & M'\ar[r]^{f'} & N'\ar[r]^{g'} & P'\ar[r]&0}$$
is exact.
 Since $\mathcal{C}$ is closed under extensions and $M^{'}, P^{'}\in{\mathcal{C}}$,
 we obtain that $N^{'}\in {\mathcal{C}}$. This shows that $N=i_{\ast}(N^{'})\in{i_{\ast}(\mathcal{C})}$ and thus $i_{\ast}(\mathcal{C})$ is closed under extensions.
	\end{proof}
	
Now we will consider the behavior of the functors $j_{\ast}$ and $j_{!}$.
	\begin{theorem}\label{3.5}
 If $j_{\ast}$ (resp. $j_!$) is exact, then $\mathcal{C}\in{\mathcal{A^{''}}_{ice}}$ implies that $j_{\ast}(\mathcal{C})\in{\mathcal{A}_{ice}}$ (resp. $j_!(\mathcal{C})\in{\mathcal{A}_{ice}}$).
	\end{theorem}
	\begin{proof}
Assume that $j_{\ast}$ is exact and $\mathcal{C}\in{\mathcal{A^{''}}_{ice}}$. We only show $j_{\ast}(\mathcal{C})\in{\mathcal{A}_{ice}}$,
and the statement on $j_!$ can be proved dually.
Since $j_{\ast}$ is fully faithful and exact,
we get that $j_{\ast}(\mathcal{C})$ is closed under images and cokernels by lemma~\ref{ho1}.
Now we claim that $j_{\ast}(\mathcal{C})$ is closed
under extensions. 		
		Let $$\xymatrix{0\ar[r] & M\ar[r]^f & N\ar[r]^g & P\ar[r]&0}$$ be an exact sequence in $\mathcal{A}$, where $M=j_{\ast}(M^{'}),P=j_{\ast}(P^{'}),M^{'}\in{\mathcal{C}}$ and $P^{'}\in{\mathcal{C}}$. Since $j^{\ast}$ is exact, we have the following exact sequence
		$$\xymatrix{0\ar[r] & j^{\ast}(M)\ar[r]^{j^{\ast}(f)} & j^{\ast}(N)\ar[r]^{j^{\ast}(g)} & j^{\ast}(P)\ar[r]&0},$$
where $j^{\ast}(M)=j^{\ast}j_{\ast}(M^{'})=M^{'}\in{\mathcal{C}}$ and $j^{\ast}(P)=j^{\ast}j_{\ast}(P^{'})=P^{'}\in{\mathcal{C}}$. Since $\mathcal{C}$ is ICE-closed, we infer that $j^{\ast}(N)\in{\mathcal{C}}$.
Note that the functor $j_*$ is exact, and then
we have the following exact commutative diagram:
			$$\xymatrix@C=3pc{ 0 \ar[r] & M   \ar[r]^{f}\ar[d]^{\eta_{M}} & N
			\ar[r]^{g}\ar[d]^{\eta_{N}} &  P \ar[r]\ar[d]^{\eta_{P}} & 0 \\
			0\ar[r]& j_{\ast}j^{\ast}(M) \ar[r]^{j_{\ast}j^{\ast}(f)}& j_{\ast}j^{\ast}(N)\ar[r]^{j_{\ast}j^{\ast}(g)} & j_{\ast}j^{\ast}(P) \ar[r]  &0,}$$
 where $\eta:\id_{\mathcal{A}}\rightarrow{j_{\ast}j^{\ast}}$ is the unit of the adjoint pair ($j^{\ast},j_{\ast}$).
Since $j^*j_*\cong \id_{\mathcal{A}''}$, we get that $M=j_{\ast}(M^{'})\cong j_*j^*j_*M'=j_*j^*M$, which shows that $\eta_{M}$ is an isomorphism.
Similarly, we obtain that $\eta_{P}$ is an isomorphism, and by Five-Lemma,
we have
$N\cong{j_{\ast}j^{\ast}(N)\in{j_{\ast}(\mathcal{C})}}$.
	\end{proof}

Next, we will investigate how ICE-closed subcategories transfer under the functors $j^{\ast}$. We need the following lemma.

	\begin{lemma}\label{3.6}
		If $\mathcal{C}\in{_{i_{\ast}(\mathcal{A^{'}})}\mathcal{A}_{ice}}$, then $j_{\ast}j^{\ast}(\mathcal{C})\subset{\mathcal{C}}$.
	\end{lemma}
	\begin{proof}
The proof is similar to that of \cite[Lemma 3.1]{Z17}. For convenience we include a short justification.
For any object $C\in\mathcal{C}\subset\mathcal{A}$,
it follows from Remark~\ref{2.4} (3) that there is an exact sequence	$$0\rightarrow i_{\ast}i^{!}(C) \rightarrow C \rightarrow j_{\ast}j^{\ast}(C)  \rightarrow i_{\ast}(M^{'}) \rightarrow0$$ for some $M^{'}\in\mathcal{A^{'}} $.
 Since $i_{\ast}(\mathcal{A^{'}})\subset\mathcal{C}$ and $\mathcal{C}$ is ICE-closed, we infer that $j_{\ast}j^{\ast}(C)\in\mathcal{C}$.
			
	\end{proof}
	\begin{proposition}\label{3.7}
		 If $\mathcal{C}\in{_{i_{\ast}(\mathcal{A^{'}})}\mathcal{A}_{ice}}$, then the following hold.
		
		{\rm (1)} If $j_!j^{\ast}(\mathcal{C})\subset{\mathcal{C}}$, then $j^{\ast}(\mathcal{C})\in{\mathcal{A^{''}}_{ice}}$.
		
		{\rm (2)} If $j_{\ast}$ is exact, then $j^{\ast}(\mathcal{C})\in{\mathcal{A^{''}}_{ice}}$.
	\end{proposition}
		\begin{proof}
(1) It follows from the proof of \cite[Proposition 3.2]{Z17} that $j^{\ast}(\mathcal{C})$ is closed under cokernels.
Now we claim that $j^{\ast}(\mathcal{C})$ is closed under images. For this, let $\xymatrix{M\ar[r]^f & N}$ be a map in $j^{\ast}(\mathcal{C})$,
and consider the map $\xymatrix{j_{\ast}(M)\ar[r]^{j_{\ast}(f)} & j_{\ast}(N)}$ in $\mathcal{A}$. By Lemma~\ref{3.6}, we get $j_{\ast}(M)\in{j_{\ast}j^{\ast}(\mathcal{C})}\subset{\mathcal{C}}$ and $j_{\ast}(N)\in{j_{\ast}j^{\ast}(\mathcal{C})}\subset{\mathcal{C}}$,
and then $\Im j_{\ast}(f)\in{\mathcal{C}}$
since $\mathcal{C}$ is ICE-closed.
Note that $j^*$ is exact and $j^{\ast}j_{\ast}\cong\id_{\mathcal{A^{''}}}$, it follows from Lemma~\ref{3.3} that $\Im f\cong{\Im j^{\ast}j_{\ast}(f)}\cong{j^{\ast}(\Im j_{\ast}(f))}\in{j^{\ast}(\mathcal{C})}$.
This shows that $j^{\ast}(\mathcal{C})$ is closed under images, and it remains to show 			
$j^{\ast}(\mathcal{C})$ is closed under extensions.

			 Let
			$$\xymatrix{0\ar[r] & M\ar[r]^f & N\ar[r]^g & P\ar[r]&0}$$
be an exact sequence in $\mathcal{A^{''}}$ with $M\in{j^{\ast}(\mathcal{C})}$ and $P\in{j^{\ast}(\mathcal{C})}$.
Since $j_!$ is right exact, we have the following exact sequence
				$$\xymatrix{0\ar[r] & \Ker j_!(g)\ar[r] & j_!(N)\ar[r]^{j_!(g)} & j_!(P)\ar[r]&0.}\eqno {\rm (3.3)}$$
				Since $j_!j^{\ast}(\mathcal{C})\subset{\mathcal{C}}$, we get that $j_!(P)\in{\mathcal{C}}$.
On the other hand, it follows from Lemma~\ref{3.3} and $j^*j_!\cong \id _{\mathcal{A}''}$ that $j^{\ast}(\Ker j_{!}(g))\cong \Ker j^*j_!(g)\cong
\Ker g\cong{M}$, which implies that $j^{\ast}(\Ker j_{!}(g))\in{j^{\ast}(\mathcal{C})}$.
By Remark~\ref{2.4} (3), there exist some $M'\in \mathcal{A}'$ and an exact sequence $$0\rightarrow i_{\ast}(M^{'}) \rightarrow j_{!}j^{\ast}(\Ker j_{!}(g)) \rightarrow \Ker j_{!}(g) \rightarrow i_{\ast}i^{\ast}(\Ker j_{!}(g)) \rightarrow 0,$$ which
shows that $\Ker j_{!}(g)\in{\mathcal{C}}$
since $j_!j^{\ast}(\Ker j_{!}(g))\in j_!j^* \mathcal{C} \subset\mathcal{C}$ and $i_*(\mathcal{A}')\subset \mathcal{C}$.
Now it follows from the exact sequence (3.3) that $j_!(N)\in{\mathcal{C}}$,
and then $N\cong{j^{\ast}j_!(N)}\in{j^{\ast}(\mathcal{C})}$. This shows that 			
$j^{\ast}(\mathcal{C})$ is closed under extensions.

(2) In the proof of (1), the assumption $j_!j^{\ast}(\mathcal{C})\subset{\mathcal{C}}$ is not used
when we deal with images and cokernels. So,
by a similar argument as in (1), we get that $j^{\ast}(\mathcal{C})$ is closed under images and cokernels.
Now we show that $j^{\ast}(\mathcal{C})$ is closed under extensions when $j_{\ast}$ is exact.
				
					 Let
					$$\xymatrix{0\ar[r] & M\ar[r]^f & N\ar[r]^g & P\ar[r]&0}$$
					be an
					exact sequence in $\mathcal{A^{''}}$ with $M\in{j^{\ast}(\mathcal{C})}$ and $P\in{j^{\ast}(\mathcal{C})}$.
					Since $j_{\ast}$ is exact, we have the following exact sequence
					$$\xymatrix{0\ar[r] & j_{\ast}(M)\ar[r]^{j_{\ast}(f)} & j_{\ast}(N)\ar[r]^{j_{\ast}(g)} & j_{\ast}(P)\ar[r]&0.}\eqno {\rm (3.4)}$$
					It follows from Lemma~\ref{3.6} that $j_{\ast}j^{\ast}(\mathcal{C})\subset{\mathcal{C}}$, and then $j_{\ast}(M)\in{\mathcal{C}}$ and $j_{\ast}(P)\in{\mathcal{C}}$. Since $\mathcal{C}$ is ICE-closed, we infer that $j_{\ast}(N)\in{\mathcal{C}}$,
which is the middle term of the exact sequence (3.4).
As a result, we have $N\cong{j^{\ast}j_{\ast}(N)}\in{j^{\ast}(\mathcal{C})}$, and thus $j^{\ast}(\mathcal{C})$ is closed under extensions.
							\end{proof}

The following proposition is essentially due to \cite[Theorem 3.5]{Z17}.
		\begin{proposition}\label{3.8}
If $\mathcal{C}\in \mathcal{A}_{ice}$ and $i_\ast{i^\ast(\mathcal{C})}\subset\mathcal{C}$ (resp. $i_\ast{i^!(\mathcal{C})}\subset\mathcal{C}$), then $i^\ast(\mathcal{C})\in \mathcal{A}'_{ice}$ (resp. $i^!(\mathcal{C})\in \mathcal{A}'_{ice}$).	
		\end{proposition}
	\begin{proof}
		Assume that $\mathcal{C}\subset{\mathcal{A}}$ is ICE-closed and $i_\ast{i^\ast(\mathcal{C})}\subset\mathcal{C}$. We only prove that ${i^\ast(\mathcal{C})}\subset{\mathcal{A^{'}}}$ is ICE-closed,
and then the statement on $i^!$ can be proved similarly.
		
	The proof that $i^{\ast}(\mathcal{C})$ is closed under cokernels and extensions is the same as \cite[Theorem 3.5]{Z17}.
	If $\xymatrix{M\ar[r]^f & N}$ is a map in $i^{\ast}(\mathcal{C})$, we only have to prove that $\Im f\in{i^{\ast}(\mathcal{C})}$. Applying $i_{\ast}$ to $f$, we have a map $\xymatrix{i_{\ast}(M)\ar[r]^{i_{\ast}(f)} & i_{\ast}(N)}$ in $\mathcal{A}$.
It follows from $i_\ast{i^\ast(\mathcal{C})}\subset\mathcal{C}$ that $i_*(f)$ is a map in $\mathcal{C}$.
Since $\mathcal{C}$ is ICE-closed,
we have $\Im i_*(f)\in \mathcal{C}$, and then $\Im f\cong{i^{\ast}i_{\ast}(\Im f)}\cong i^*(\Im i_*(f))\in{i^{\ast}(\mathcal{C})}$.
	\end{proof}

Next, we will show that starting from an ICE-closed subcategory of $\mathcal{A}''$, we can also find a corresponding
ICE-closed subcategory of $\mathcal{A}$ via the functor $j^*$, as shown in the following lemma.

	\begin{lemma}\label{3.9}
Let $\xymatrix{\mathcal{A}\ar[r]^F & \mathcal{B}}$ be a functor between abelian categories. If $F$ is exact and $\mathcal{W}\in{\mathcal{B}_{ice}}$,
 then $\mathcal{C}=\left\lbrace {M\in{\mathcal{A}}|F(M)\in \mathcal{W}}\right\rbrace \in{\mathcal{A}_{ice}}$.
	\end{lemma}
	\begin{proof}
This can be proved similarly as \cite[proposition 3.3]{Z17}.	
	\end{proof}
	
Now we get the following consequence, which is the converse of Proposition~\ref{3.7} and Proposition~\ref{3.8}.
	
	\begin{corollary}\label{3.10}
		
{\rm (1)} If $\mathcal{W}\in{\mathcal{A}^{''}}_{ice}$, then $\mathcal{C}=\left\lbrace {M\in{\mathcal{A}}\ |\ j^{\ast}(M)\in{\mathcal{W}}}\right\rbrace \in{_{i_{\ast}(\mathcal{A^{'}})}\mathcal{A}_{ice}}$.

{\rm (2)} If $i^{\ast}$ is exact and $\mathcal{W}\in{\mathcal{A}^{'}}_{ice}$, then $\mathcal{C}=\left\lbrace {M\in{\mathcal{A}}\ |\ i^{\ast}(M)\in{\mathcal{W}}}\right\rbrace \in{\mathcal{A}_{ice}}$.

{\rm (3)} If $i^!$ is exact and $\mathcal{W}\in{\mathcal{A}^{'}}_{ice}$, then $\mathcal{C}=\left\lbrace {M\in{\mathcal{A}}\ |\ i^!(M)\in{\mathcal{W}}}\right\rbrace \in{\mathcal{A}_{ice}}$.
\end{corollary}
		\begin{proof}
 In view of Lemma~\ref{3.9}, it suffices to show $i_*(\mathcal{A}')\subset \mathcal{C}$ in the case of (1).
Indeed, this follows from $j^{\ast}i_{\ast}=0$ and the fact that the zero object, which is the cokernel of the identity
map, always belongs to $\mathcal{W}$.   			
		\end{proof}

In conclusion, we have the following theorem, which is an analogy of \cite[Theorem 3.4]{Z17}.
		\begin{theorem}\label{thm-bij}
There is a bijection $$\xymatrix@!=4pc{\{\mathcal{C}\in {_{i_{\ast}(\mathcal{A^{'}})}\mathcal{A}_{ice}}
\ | \  j_!j^*\mathcal{C}\subset \mathcal{C}\}\ar[rr]^-{1:1} & &\ar[ll]\mathcal{A''}_{ice}}$$
which sends any $\mathcal{C}\in{_{i_{\ast}(\mathcal{A^{'}})}\mathcal{A}_{ice}}$ to $j^*(\mathcal{C})$,
and sends any $\mathcal{W}\in{\mathcal{A}^{''}}_{ice}$ to $\left\lbrace {M\in{\mathcal{A}}\ |\ j^{\ast}(M)\in{\mathcal{W}}}\right\rbrace$.
		\end{theorem}
		\begin{proof}
For any $\mathcal{C}\in{_{i_{\ast}(\mathcal{A^{'}})}\mathcal{A}_{ice}}$
with $j_!j^{\ast}(\mathcal{C})\subset{\mathcal{C}}$, it follows from Proposition~\ref{3.7} (1)
that $j^{\ast}(\mathcal{C})\in{\mathcal{A^{''}}_{ice}}$. Conversely, for any
 $\mathcal{W}\in{\mathcal{A}^{''}}_{ice}$, we have
 $\left\lbrace {M\in{\mathcal{A}}\ |\ j^{\ast}(M)\in{\mathcal{W}}}\right\rbrace \in{_{i_{\ast}(\mathcal{A^{'}})}\mathcal{A}_{ice}}$
by Corollary~\ref{3.10} (1). Denote by $\mathcal{C}':=\left\lbrace {M\in{\mathcal{A}}\ |\ j^{\ast}(M)\in{\mathcal{W}}}\right\rbrace $,
and we have to prove $j_!j^{\ast}(\mathcal{C}')\subset{\mathcal{C}'}$. But this is trivial
because $j^*j_!j^*\mathcal{C}'=j^*\mathcal{C}'\subset \mathcal{W}$.
Therefore, the maps in this theorem are well-defined, and we need to show they are inverse to each other.
Indeed, this can be
done by using the same argument as \cite[Theorem 3.4]{Z17}.
		\end{proof}

We mention that
the assumption $j_!j^{\ast}(\mathcal{C})\subset{\mathcal{C}}$ in Theorem~\ref{thm-bij} is necessary,
while it is not needed in \cite[Theorem 3.4]{Z17}.
The reason is that $j_!j^{\ast}(\mathcal{C})\subset{\mathcal{C}}$ always hold true
for any wide subcategory $\mathcal{C}$, see \cite[Lemma 3.1]{Z17}.
But this is not the case for ICE-closed subcategory.
As an analogy of \cite[Theorem 3.8]{Z17}, we have the following conclusion with the additional condition
that $j_!j^{\ast}(\mathcal{C})\subset{\mathcal{C}}$.
		
	\begin{theorem}\label{thm-recoll}
If $\mathcal{C}\in{_{i_{\ast}(\mathcal{A^{'}})}\mathcal{A}_{ice}}$
and $j_!j^{\ast}(\mathcal{C})\subset{\mathcal{C}}$, then we get a new recollement of ICE-closed subcategories $(  \mathcal{A^{'}},\mathcal{C},j^\ast(\mathcal{C}),\overline{i^\ast},\overline{i_\ast},\overline{i^!},\overline{j_!},\overline{j^\ast},\overline{j_\ast})$
induced by the restriction functors. 		
	\end{theorem}
	
	\begin{proof}
By Lemma~\ref{3.7} (1), we get that $j^*(\mathcal{C})$ is an ICE-closed subcategory.
Since $i_*i^*(\mathcal{C})\subset \mathcal{C}$ and $i_*i^!(\mathcal{C})\subset \mathcal{C}$,
it follows from \cite[Proposition 3.6]{Z17} and \cite[Lemma 3.7]{Z17} that the functors
$i^*,i_*$ and $i^!$ can be restricted to
adjoint pairs between the corresponding subcategories, and so are
the functors $j_!,j^*$ and $j_*$, because $j_!j^{\ast}(\mathcal{C})\subset{\mathcal{C}}$ and $j_*j^{\ast}(\mathcal{C})\subset{\mathcal{C}}$,
see Lemma~\ref{3.6}.
The rest of the proof is the same as \cite[theorem 3.8]{Z17}.
	\end{proof}
	
Let's end this section by a similar discussion on torsion classes.
	\begin{corollary}\label{cor-tor}	
		{\rm (1)} If $\mathcal{C}\in{\mathcal{A^{'}}_{tor}}$, then $i_{\ast}(\mathcal{C})\in{\mathcal{A}_{tor}}$.

		{\rm (2)} If $i^!$ (resp. $i^{\ast}$) is exact and $\mathcal{C}\in{\mathcal{A^{''}}_{tor}}$,
 then $j_{\ast}(\mathcal{C})\in{\mathcal{A}_{tor}}$ (resp. $j_!(\mathcal{C})\in{\mathcal{A}_{tor}}$).

		{\rm (3)} If $\mathcal{C}\in{_{i_{\ast}(\mathcal{A^{'}})}\mathcal{A}_{tor}}$, then the following hold.
		
		{\rm (a)} If $j_!j^{\ast}(\mathcal{C})\subset{\mathcal{C}}$, then $j^{\ast}(\mathcal{C})\in{\mathcal{A^{''}}_{tor}}$.
		
		{\rm (b)} If $j_{\ast}$ is exact, then $j^{\ast}(\mathcal{C})\in{\mathcal{A^{''}}_{tor}}$.

		{\rm (4)} If $\mathcal{C}\in \mathcal{A}_{tor}$ and $i_\ast{i^\ast(\mathcal{C})}\subset\mathcal{C}$ (resp. $i_\ast{i^!(\mathcal{C})}\subset\mathcal{C}$), then $i^\ast(\mathcal{C})\in \mathcal{A}'_{tor}$ (resp. $i^!(\mathcal{C})\in \mathcal{A}'_{tor}$).

		{\rm (5)} If $\mathcal{W}\in{\mathcal{A}^{''}}_{tor}$, then $\mathcal{C}=\left\lbrace {M\in{\mathcal{A}}\ |\ j^{\ast}(M)\in{\mathcal{W}}}\right\rbrace \in{_{i_{\ast}(\mathcal{A^{'}})}\mathcal{A}_{tor}}$.
	
		{\rm (6)} If $i^{\ast}$ is exact and $\mathcal{W}\in{\mathcal{A}^{'}}_{tor}$, then $\mathcal{C}=\left\lbrace {M\in{\mathcal{A}}\ |\ i^{\ast}(M)\in{\mathcal{W}}}\right\rbrace \in{\mathcal{A}_{tor}}$.
		
		{\rm (7)} If $i^!$ is exact and $\mathcal{W}\in{\mathcal{A}^{'}}_{tor}$, then $\mathcal{C}=\left\lbrace {M\in{\mathcal{A}}\ |\ i^!(M)\in{\mathcal{W}}}\right\rbrace \in{\mathcal{A}_{tor}}$.
		
	\end{corollary}
\begin{proof}
	(1) For any $\mathcal{C}\in{\mathcal{A^{'}}_{tor}}$, we have that $\mathcal{C}\in{\mathcal{A^{'}}_{ice}}$
and then $i_{\ast}(\mathcal{C})\in{\mathcal{A}_{ice}}$ by Theorem~\ref{3.4}.
Therefore, we get that $i_{\ast}(\mathcal{C})$ is closed under extensions, and we only need to show
$i_{\ast}(\mathcal{C})$ is closed under quotients.
Let $\xymatrix{ M\ar[r] & N\ar[r]&0}$ be an exact sequence in $\mathcal{A}$ with $M=i_{\ast}(M^{'})$ and $M^{'}\in{\mathcal{C}}$. Since $j^{\ast}$ is exact,
we get an exact sequence $\xymatrix{ j^{\ast}(M)\ar[r] & j^{\ast}(N)\ar[r]&0}$ in $\mathcal{A}''$.
It follows from $\Im i_*=\Ker j^*$ that $j^*M=j^*i_*M'=0$, and then
$j^{\ast}(N)=0$.
Therefore, we infer $N\in \Ker j^*$ and then $N=i_{\ast}(N^{'})$ for some $N^{'}\in{\mathcal{A^{'}}}$.
Since $i^{\ast}$ is right exact, we get an exact sequence $\xymatrix{ i^{\ast}(M)\ar[r] & i^{\ast}(N)\ar[r]&0}$,
where $i^*M=i^*i_*M'\cong M'\in \mathcal{C}$. Now note that $\mathcal{C}$ is closed under quotients,
and then
$i^*N\in \mathcal{C}$. Therefore, we get that $N'\cong i^*i_*N'=i^*N\in \mathcal{C}$, and thus $N=i_{\ast}(N^{'})\in{i_{\ast}(\mathcal{C})}$.
	
(2) Assume that $i^!$ is exact. Then it follows from
\cite[Lemma 3.1 (4)]{FZ17} or \cite[Theorem 4.1]{TO79} that $j_{\ast}$ is exact
and $\Im j_{\ast}=\Ker i^!$.
Therefore, for any $\mathcal{C}\in{\mathcal{A^{''}}_{tor}}$,
we get that $j_*\mathcal{C}$ is closed under quotients in a similar way as we did in (1).
Moreover, it follows from Theorem~\ref{3.5} that $j_*\mathcal{C}$ is closed under extensions.
As a result, we get that $j_{\ast}(\mathcal{C})\in{\mathcal{A}_{tor}}$, and the statement on
$j_!(\mathcal{C})$ can be proved similarly.
	
	(3) For any $\mathcal{C}\in{_{i_{\ast}(\mathcal{A^{'}})}\mathcal{A}_{tor}}$,
it follows from Proposition~\ref{3.7} that $j^{\ast}(\mathcal{C})$ is closed under extensions.
Now we will claim $j^{\ast}(\mathcal{C})$ is closed under quotients in the following two cases.
	
	(a) Assume that $j_!j^{\ast}(\mathcal{C})\subset{\mathcal{C}}$,
and let $\xymatrix{ M\ar[r] & N\ar[r]&0}$ be an exact sequence in $\mathcal{A^{''}}$ with $M\in{j^{\ast}(\mathcal{C})}$. Since $j_{!}$ is right exact,
 we get an exact sequence $\xymatrix{ j_!(M)\ar[r] & j_!(N)\ar[r]&0}$, where $j_!M \in j_!j^{\ast}(\mathcal{C})\subset{\mathcal{C}}$.
 Since $\mathcal{C}$ is closed under quotients,
 we obtain that $j_!(N)\in{\mathcal{C}}$.
 Therefore, we get $N\cong{j^{\ast}j_!(N)}\in{j^{\ast}(\mathcal{C})}$
 and then $j^{\ast}(\mathcal{C})$ is closed under quotients.
	
	(b) Assume $j_{\ast}$ is exact, and let $\xymatrix{ M\ar[r] & N\ar[r]&0}$ be an exact
sequence in $\mathcal{A^{''}}$ with $M\in{j^{\ast}(\mathcal{C})}$.
Then the sequence $\xymatrix{ j_{\ast}(M)\ar[r] & j_{\ast}(N)\ar[r]&0}$is also exact.
Since $\mathcal{C}\in{_{i_{\ast}(\mathcal{A^{'}})}\mathcal{A}_{tor}}\subset
{_{i_{\ast}(\mathcal{A^{'}})}\mathcal{A}_{ice}}$,
it follows from lemma~\ref{3.6} that $j_{\ast}j^{\ast}(\mathcal{C})\subset{\mathcal{C}}$.
Therefore, we have that $j_*M\in \mathcal{C}$,
and thus
$j_{\ast}(N)\in{\mathcal{C}}$
since $\mathcal{C}$ is closed under quotients.
Hence, $N\cong{j^{\ast}j_{\ast}}(N)\in{j^{\ast}(\mathcal{C})}$ and $j^{\ast}(\mathcal{C})$ is closed under quotients.
	
(4) For any $\mathcal{C}\in \mathcal{A}_{tor}$ with $i_\ast{i^\ast(\mathcal{C})}\subset\mathcal{C}$,
it follows from Proposition~\ref{3.8} that $i^{\ast}(\mathcal{C})$ is closed under extensions.
Now we claim that $i^{\ast}(\mathcal{C})$ is closed under quotients.
Let $\xymatrix{ M\ar[r] & N\ar[r]&0}$ be an exact sequence in $\mathcal{A^{'}}$ with $M\in{i^{\ast}(\mathcal{C})}$.
Applying the exact functor $i_{\ast}$,
we get an exact sequence $\xymatrix{ i_{\ast}(M)\ar[r] & i_{\ast}(N)\ar[r]&0}$,
where $i_{\ast}M \in i_{\ast}{i^{\ast}(\mathcal{C})}\subset\mathcal{C}$.
Since $\mathcal{C}$ is closed under quotients, we infer that $i_{\ast}(N)\in\mathcal{C}$, and thus
$N\cong{i^{\ast}i_{\ast}(N)}\in{i^{\ast}(\mathcal{C})}$.
The proof that $i^!(\mathcal{C})\in \mathcal{A}'_{tor}$ when
 $i_\ast{i^!(\mathcal{C})}\subset\mathcal{C}$ is similar.

(5) Let $\mathcal{W}\in{\mathcal{A}^{''}}_{tor}$ and $\mathcal{C}=\left\lbrace {M\in{\mathcal{A}}\ |\ j^{\ast}(M)\in{\mathcal{W}}}\right\rbrace$.
Then it follows from Corollary~\ref{3.10} (1) that $\mathcal{C}$ is closed under extensions
and $i_*(\mathcal{A}')\subset \mathcal{C}$. Now we claim that $\mathcal{C}$ is closed under quotients.
Let $\xymatrix{ M\ar[r] & N\ar[r]&0}$ be an exact sequence in $\mathcal{A}$ with $M\in{\mathcal{C}}$. Then we get $j^{\ast}(M)\in{\mathcal{W}}$.
Applying the exact functor $j^{\ast}$, we get an exact sequence $\xymatrix{ j^{\ast}(M)\ar[r] & j^{\ast}(N)\ar[r]&0}$.
 Since $\mathcal{W}$ closed under quotients, we infer that $j^{\ast}(N)\in{\mathcal{W}}$,
and then $N\in{\mathcal{C}}$ by the definition of $\mathcal{C}$.

(6) and (7) can be proved in a similar way as we did in (5).

\end{proof}

	\section{Epibrick over recollements}\label{4}
	\indent\indent
In \cite{Z19}, it is shown that semibricks can be glued via recollements, and in this section,
we will observe
how to glue epibricks and monobricks.
The following lemma is an analogy of
\cite[Lemma 2.1]{Z19}, which says that both bricks and semibricks are preserved under fully faithful functors.
	
	\begin{lemma}\label{4.1}
Let $F:\mathcal{A}\rightarrow{\mathcal{B}}$ be a fully faithful functor
between abelian categories. If $F$ preserves epimorphisms (resp. monomorphisms), then
we have $F(\ebrick \mathcal{A})\subset \ebrick \mathcal{B}$ (resp.
$F(\mbrick \mathcal{A})\subset \mbrick \mathcal{B}$).
	\end{lemma}
	\begin{proof}
Assume that $F$ is fully faithful and $F$ preserves epimorphisms,
and let $\mathcal{C}\in \ebrick \mathcal{A}$.
It follows from \cite[Lemma 2.1]{Z19} that $F(\mathcal{C})\subset \brick \mathcal{B}$.
Further, for any $h\in{\Hom _{\mathcal{B}}(F(C_1),F(C_2))}$ with $C_1,C_2 \ {\in\mathcal{C}}$, there is some
$h_1\in{\Hom_{\mathcal{A}}(C_1,C_2)}$ such that $h=F(h_1)$.
Therefore, we get $h_1\neq 0$ for any $h\neq 0$, and then $h_1$
is an epimorphism, for the reason that $\mathcal{C}$ is an epibrick.
Since $F$ preserves epimorphisms, we infer that $h$
is an epimorphism and thus $F(\mathcal{C})\in \ebrick \mathcal{B}$.
The statement on monobricks can be proved dually.	
	\end{proof}

Applying Lemma~\ref{4.1} to recollement, we get the following theorem,
which is an analogy of
\cite[Proposition 2.1]{Z19}.
		
	\begin{theorem}\label{4.3} Let $(  \mathcal{A^{'}},\mathcal{A},\mathcal{A^{''}},i^\ast,i_\ast,i^!,j_!,j^\ast,j_\ast)$ be a
recollement of abelian categories.
Then we have
	
{\rm (1)} $i_*(\ebrick \mathcal{A}')\subset \ebrick \mathcal{A}$ and
$i_*(\mbrick \mathcal{A}')\subset \mbrick \mathcal{A}$;

{\rm (2)} $j_{!*}(\ebrick \mathcal{A}')\subset \ebrick \mathcal{A}$ and
$j_{!*}(\mbrick \mathcal{A}')\subset \mbrick \mathcal{A}$;
	
	 {\rm (3)} $j_{*}(\mbrick \mathcal{A}'')\subset \mbrick \mathcal{A}$;
	
	 {\rm (4)} $j_{!}(\ebrick \mathcal{A}'')\subset \ebrick \mathcal{A}$.
	\end{theorem}
\begin{proof}
By Remark~\ref{2.4} (1) and Lemma~\ref{lem-int-func} (2), we have that both $i_*$
and $j_{!*}$ preserve epimorphisms and monomorphisms,
$j_*$ preserves monomorphisms and $j_!$ preserves epimorphisms. Further, it follows from
the definition of recollement and Lemma~\ref{lem-int-func} (1) that $i_*, j_{!*}, j_{*}$
and $j_{!}$ are fully faithful. Now this theorem follows from Lemma~\ref{4.1} immediately.
\end{proof}		

 Motivated by gluing simples in \cite{BBD82} and semibricks in \cite{Z19}, here we also give a
 construction of monobricks (resp. epibricks) from the left and right side into the middle.

 \begin{theorem}\label{4.4}
Let $(  \mathcal{A^{'}},\mathcal{A},\mathcal{A^{''}},i^\ast,i_\ast,i^!,j_!,j^\ast,j_\ast)$ be a recollement of abelian categories.
Then we have $i_*(\mbrick \mathcal{A}')\cup j_{!*}(\mbrick \mathcal{A}'')\subset \mbrick \mathcal{A}$
and $i_*(\ebrick \mathcal{A}')\cup j_{!*}(\ebrick \mathcal{A}'')\subset \ebrick \mathcal{A}$.	
	\end{theorem}
	\begin{proof}
Let $\mathcal{C}'\in \mbrick \mathcal{A}'$ and $\mathcal{C}''\in \mbrick \mathcal{A}''$. By Theorem~\ref{4.3},
we have $i_*(\mathcal{C}')\in \mbrick \mathcal{A}$ and $j_{!*}(\mathcal{C}'')\in \mbrick \mathcal{A}$.
To prove that $i_*(\mathcal{C}')\cup j_{!*}(\mathcal{C}'')\in \mbrick \mathcal{A}$,
it suffices to show every non-zero map from $i_*(\mathcal{C}')$ to $j_{!*}(\mathcal{C}'')$
or from $j_{!*}(\mathcal{C}'')$ to $i_*(\mathcal{C}')$ is either zero or
monomorphic. Indeed, it follows from Lemma~\ref{lem-int-func} (3) that $i^*j_{!*}=0$ and
$i^!j_{!*}=0$, and then $\Hom _{\mathcal{A}}(i_*(\mathcal{C}'), j_{!*}(\mathcal{C}''))\cong
\Hom _{\mathcal{A}'}(\mathcal{C}', i^!j_{!*}(\mathcal{C}''))\cong 0$ and
$\Hom _{\mathcal{A}}(j_{!*}(\mathcal{C}''), i_*(\mathcal{C}'))\cong
\Hom _{\mathcal{A}'}( i^*j_{!*}(\mathcal{C}''),\mathcal{C}' )\cong 0$.
Thus $i_*(\mathcal{C}')\cup j_{!*}(\mathcal{C}'')\in \mbrick \mathcal{A}$, and the statement on epibricks
can be proved dually.
	\end{proof}

Since $i^!j_!$ and $i^*j_*$ may not be zero, $i_*(\mbrick \mathcal{A}')\cup j_{!}(\mbrick \mathcal{A}'')$
and $i_*(\mbrick \mathcal{A}')\cup j_{*}(\mbrick \mathcal{A}'')$ will not usually
 be monobricks in the middle category. But we have the following special case due to
 \cite[Proposition 8.8]{FP04}.

	\begin{theorem}\label{4.5}
If $i^*$ is exact, then we have $i_*(\mbrick \mathcal{A}')\cup j_{!}(\mbrick \mathcal{A}'')\subset \mbrick \mathcal{A}$
and $i_*(\ebrick \mathcal{A}')\cup j_{!}(\ebrick \mathcal{A}'')\subset \ebrick \mathcal{A}$.
Dually, if $i^!$ is exact, then we have $i_*(\mbrick \mathcal{A}')\cup j_{*}(\mbrick \mathcal{A}'')\subset \mbrick \mathcal{A}$
and $i_*(\ebrick \mathcal{A}')\cup j_{*}(\ebrick \mathcal{A}'')\subset \ebrick \mathcal{A}$.	
	\end{theorem}
	\begin{proof}
We prove the first assertion, and the second can be proved dually.
Assume that $i^*$ is exact. Then it follows from \cite[Lemma 3.2]{FZ17} that
$j_!$ is exact, and thus it preserves monomorphisms. Applying Lemma~\ref{4.1},
we have that $j_{!}(\mbrick \mathcal{A}'')\subset \mbrick \mathcal{A}$.
Therefore,
to show $i_*(\mbrick \mathcal{A}')\cup j_{!}(\mbrick \mathcal{A}'')\subset \mbrick \mathcal{A}$,
it suffices to show $\Hom _{\mathcal{A}}(\Im i_*, \Im j_{!})\cong 0$
and $\Hom _{\mathcal{A}}(\Im j_{!}, \Im i_*)\cong 0$, and then the proof
is finished by using the same argument as Theorem~\ref{4.4}.
Since $i^*$ is exact, we get
$i^!j_!=0$ by \cite[Proposition 8.8]{FP04},
and then $\Hom _{\mathcal{A}}(\Im i_*, \Im j_{!})\cong 0$ by adjointness.
Further, it follows from $j^*i_*=0$ that
$\Hom _{\mathcal{A}}(\Im j_{!}, \Im i_*)\cong 0$.
Thus we conclude that $i_*(\mbrick \mathcal{A}')\cup j_{!}(\mbrick \mathcal{A}'')\subset \mbrick \mathcal{A}$, and the statement on epibricks can be
proved similarly.	
	\end{proof}

	\section{Applications and examples}\label{5}
	\indent\indent
	Throughout this section, all algebras are basic, connected, finite dimensional $k$-algebras
over a field $k$.
Let $A$ be such an algebra, and let $\mod A$ be the category of finitely generated left $A$-modules.
We denote by $\add M$ the subcategory consisting of direct summands of finite direct sums of $M$ for $M\in{\mod A}$.
In this section, we will apply our main results to triangular matrix algebras, one-point extension algebras
and Morita rings.
	
Let $A$ and $B$ be two algebras, $_{A}M_B$ an $A$-$B$-bimodule, and $\varLambda=\left( \begin{array} {cc}
	A	&M  \\
0		& B
	\end{array}\right) $. Let $G=M\otimes_{B}-$. A left $\Lambda$-module is identified with a triple $\left(\begin{array}{c}
		X	\\Y
	\end{array} \right)_{\phi}$, where $X\in{\mod A},Y\in{\mod B}$, and $\phi:M\otimes_{B}Y\rightarrow{X}$ is a morphism of $A$-modules.
By \cite[Theorem 3.1]{Lu17}
or \cite[Example 2.12]{Psa14},
we have the following recollement of abelian categories:
$$\xymatrix@!=4pc{ \mod A \ar[r]|{i_{\ast}} & \mod\Lambda \ar@<-3ex>[l]_{i^{\ast}}
	\ar@<+3ex>[l]^{i^!} \ar[r]|{j^{\ast}} & \mod B
	\ar@<-3ex>[l]_{j_!} \ar@<+3ex>[l]^{j_{\ast}}}  ,$$
	where $i^{\ast}$ is given by $\left(\begin{array}{c}
		X	\\Y
	\end{array} \right)_{\phi}\mapsto{\Coker(\phi)}$; $i_{\ast}$ is given by $X\mapsto{\left(\begin{array}{c}
		X	\\0
	\end{array} \right)_{0}}$; $i^{!}$ is given by $\left(\begin{array}{c}
	X	\\Y
\end{array} \right)_{\phi}\mapsto{X}$; $j_!$ is given by $Y\mapsto{\left(\begin{array}{c}
	G(Y)	\\Y
\end{array} \right)_{\id}}$; $j^{\ast}$ is given by $\left(\begin{array}{c}
X	\\Y
\end{array} \right)_{\phi}\mapsto{Y}$; $j_{\ast}$ is given by $Y\mapsto{\left(\begin{array}{c}
0	\\Y
\end{array} \right)_{0}}$.

It is clear that both $i^!$ and $j_*$ are exact, and then
Theorem~\ref{3.4}, ~\ref{3.5}, ~\ref{3.7} (2), ~\ref{3.10}, ~\ref{thm-bij}, ~\ref{thm-recoll},
~\ref{cor-tor}, ~\ref{4.3} and ~\ref{4.5} apply. That is, we can construct ICE-closed subcategories, epibricks
and monobricks
of a triangular matrix algebra by its corner algebras.
In particular, let $\varLambda=\left( \begin{array} {cc}
	A	&M  \\
0		& k
	\end{array}\right) $ be the one-point extension algebra of $A$ by
an $A$-$k$-bimodule $_{A}M_k$.
Then all
left $\Lambda$-modules can be seen as $\left(\begin{array}{c}
		X	\\k^n
	\end{array} \right)_{\phi}$, where $X\in{\mod A},n\in \mathbb{N}$, and $\phi:M\otimes_{k}k^n\rightarrow{X}$ is a morphism of $A$-modules.
In this case, we recover the main theorem of \cite{LG24}.

\begin{corollary}\label{cor-one-exten}{\rm (\cite[Theorem 1.1 and 1.2]{LG24})}
Let $A$ be an algebra, $_{A}M_k$ an $A$-$k$-bimodule and $\varLambda=\left( \begin{array} {cc}
	A	&M  \\
0		& k
	\end{array}\right).$
Let $\mathcal{C}_A$ be an ICE-closed subcategory in $\mod A$, and $\mathcal{S}_A$ an epibrick in $\mod A$.
Then

{\rm (1)} $\left\lbrace \left(\begin{array}{c}
	_{A}X	\\0
\end{array} \right)_{0}|\ _AX\in{\mathcal{C}_A}\right\rbrace $  is an ICE-closed subcategory in $\mod \varLambda$.
	
{\rm (2)} $\left\lbrace \left( \begin{array}{c}
	_{A}X	\\k^n
\end{array} \right)_{f}|\ _AX\in{\mathcal{C}_A}, n\in \mathbb{N} \right\rbrace $ is an ICE-closed subcategory in $\mod \varLambda$.

{\rm (3)} $\left\lbrace{\left(\begin{array}{c}
		_AX	\\0
	\end{array} \right)_{0}|\ _AX\in{\mathcal{S}_A} } \right\rbrace $ is an epibrick subcategory in $\mod \varLambda$.
	
{\rm (4)} $\left\lbrace{\left(\begin{array}{c}
		_AX	\\0
	\end{array} \right)_{0},   \left(\begin{array}{c}
		0	\\k
	\end{array}\right)_{0}|\ _AX\in{\mathcal{S}_A} } \right\rbrace $ is an epibrick subcategory in $\mod \varLambda$.
\end{corollary}
\begin{proof}
This follows from Theorem~\ref{3.4}, Corollary~\ref{3.10} (3), Theorem~\ref{4.3} (1)
and Theorem~\ref{4.5}.
\end{proof}

Now we give an easy example to illustrate our result.

\begin{example}{\rm
Let $\Lambda:=kQ_{\Lambda}$ be a path algebra with
$Q_{\Lambda}:$ $\xymatrix{ 4\ar[r]^{\alpha} & 1\ar[r]^{\beta}&2\ar[r]^{\gamma}&3}$.
Let $A:=kQ_{A}$ with $Q_{A}:$$\xymatrix{ 4\ar[r]^{\alpha} & 1}$, and $B:=kQ_{B}$ with $Q_{B}:$$\xymatrix{2\ar[r]^{\gamma}&3}$.
Then $\Lambda =\left({\begin{array}{cc}
			A& _{A}M_B \\
			0& B
	\end{array}} \right) $ with $M=e_A\Lambda e_B$. The irreducible representations of $A$ are $1,4,\begin{array}{cc}
	4 \\
	1
\end{array}$, and the ICE-closed subcategories in $\mod A$ are $$\add\left\lbrace {0}\right\rbrace, \add\left\lbrace {1}\right\rbrace,\add\left\lbrace {4}\right\rbrace, \add\left\lbrace {\begin{array}{cc}
	4 \\
	1
	\end{array} }\right\rbrace, \add\left\lbrace {4,1,\begin{array}{cc}
	4 \\
	1
\end{array} }\right\rbrace,\add\left\lbrace {4,\begin{array}{cc}
4 \\
1
\end{array} }\right\rbrace.$$	
The irreducible representations of $B$ are $2,3,\begin{array}{cc}
	2 \\
	3
\end{array}$, and the ICE-closed subcategories in $\mod B$ are  $$\add\left\lbrace {0}\right\rbrace, \add\left\lbrace {2}\right\rbrace,\add\left\lbrace {3}\right\rbrace, \add\left\lbrace {\begin{array}{cc}
	2 \\
	3
	\end{array} }\right\rbrace,\add\left\lbrace {2,3,\begin{array}{cc}
	2 \\
	3
\end{array} }\right\rbrace,\\
\add\left\lbrace {2,\begin{array}{cc}
2 \\
3
\end{array} }\right\rbrace.$$

(1) By theorem~\ref{3.4}, we construct ICE-closed
 subcategories of $\Lambda$ by those of $A$:		
$$\begin{tabular}{|c|c|}
	\hline
$\mod A_{ice}$	& $\mod\Lambda_{ice}$ \\
	\hline
$\add\left\lbrace {0}\right\rbrace$	&$\add\left\lbrace {0}\right\rbrace$  \\
	\hline
$\add\left\lbrace {1}\right\rbrace$	&$\add\left\lbrace {1}\right\rbrace$  \\
	\hline
$\add\left\lbrace {4}\right\rbrace$	& $\add\left\lbrace {4}\right\rbrace$ \\
	\hline
$\add\left\lbrace {\begin{array}{cc}
		4 \\
		1
\end{array} }\right\rbrace$	&$\add\left\lbrace {\begin{array}{cc}
4 \\
1
\end{array} }\right\rbrace$  \\
	\hline
$\add\left\lbrace {4,1,\begin{array}{cc}
		4 \\
		1
\end{array} }\right\rbrace$	& $\add\left\lbrace {4,1,\begin{array}{cc}
4 \\
1
\end{array} }\right\rbrace$ \\
	\hline
$\add\left\lbrace {4,\begin{array}{cc}
		4 \\
		1
\end{array} }\right\rbrace $	&$\add\left\lbrace {4,\begin{array}{cc}
4 \\
1
\end{array} }\right\rbrace $  \\
	\hline
\end{tabular}$$	
		
(2) By Corollary~\ref{3.10} (3), we have the following correspondence:
\begin{center}
\begin{longtable}{|c|c|}
	\hline
$\mod A_{ice}$	& $\mod\Lambda_{ice}$ \\
	\hline\endfirsthead
	\hline
	$\mod A_{ice}$	& $\mod\Lambda_{ice}$ \\
	\hline\endhead
$\add\left\lbrace {0}\right\rbrace$	& $\add\left\lbrace {2,3,\begin{array}{cc}
		2\\
		3
\end{array}}\right\rbrace$ \\
	\hline
$\add\left\lbrace {1}\right\rbrace$	& $\add\left\lbrace {2,3,\begin{array}{cc}
		2\\
		3
	\end{array},1,\begin{array}{cc}
1 \\
2
\end{array},\begin{array}{cc}
1 \\
2\\
3
\end{array}}\right\rbrace$  \\
	\hline
$\add\left\lbrace {4}\right\rbrace$	& $\add\left\lbrace {2,3,4,\begin{array}{cc}
		2\\
		3
	\end{array}}\right\rbrace$ \\
	\hline
$\add\left\lbrace {\begin{array}{cc}
		4 \\
		1
\end{array} }\right\rbrace $	& $\add\left\lbrace {2,3,\begin{array}{cc}
2\\
3
\end{array},\begin{array}{cc}
4\\
1
\end{array},\begin{array}{cc}
4\\
1\\
2
\end{array},\begin{array}{cc}
4\\
1\\
2\\
3
\end{array}}\right\rbrace$ \\
	\hline
$\add\left\lbrace {4,\begin{array}{cc}
		4 \\
		1
\end{array} }\right\rbrace $	& $\add\left\lbrace {2,3,4,\begin{array}{cc}
2\\
3
\end{array},\begin{array}{cc}
4\\
1
\end{array},\begin{array}{cc}
4\\
1\\
2
\end{array},\begin{array}{cc}
4\\
1\\
2\\
3
\end{array}}\right\rbrace$  \\
	\hline
$\add\left\lbrace {4,1,\begin{array}{cc}
		4 \\
		1
\end{array} }\right\rbrace$		& $\add\left\lbrace {1,2,3,4,\begin{array}{cc}
2\\
3
\end{array},\begin{array}{cc}
4\\
1
\end{array},\begin{array}{cc}
4\\
1\\
2
\end{array},\begin{array}{cc}
4\\
1\\
2\\
3
\end{array},\begin{array}{cc}
1 \\
2
\end{array},\begin{array}{cc}
1 \\
2\\
3
\end{array}}\right\rbrace$  \\
	\hline
	
\end{longtable}	
\end{center}	
(3) By theorem~\ref{3.5}, we construct ICE-closed
 subcategories of $\Lambda$ by those of $B$:		
\begin{center}		
		\begin{longtable}{|c|c|}
			\hline
	$	\mod B_{ice}$	&$\mod\Lambda_{ice}$  \\
			\hline\endfirsthead
			\hline
			$	\mod B_{ice}$	&$\mod\Lambda_{ice}$  \\
			\hline\endhead
		$\add\left\lbrace {0}\right\rbrace$&$\add\left\lbrace {0}\right\rbrace$\\
		\hline
			$\add\left\lbrace {2}\right\rbrace$&$\add\left\lbrace {2}\right\rbrace$  \\
			\hline
			$\add\left\lbrace {3}\right\rbrace$&$\add\left\lbrace {3}\right\rbrace$  \\
			
			\hline
		$\add\left\lbrace {\begin{array}{cc}
				2 \\
				3
		\end{array} }\right\rbrace$	&$\add\left\lbrace {\begin{array}{cc}
				2 \\
				3
		\end{array}  }\right\rbrace$  \\
		\hline
		$\add\left\lbrace {2,\begin{array}{cc}
				2 \\
				3
		\end{array} }\right\rbrace$	&$\add\left\lbrace {2,\begin{array}{cc}
				2 \\
				3
		\end{array} }\right\rbrace$  \\
		\hline
		
		$\add\left\lbrace {2,3,\begin{array}{cc}
				2 \\
				3
		\end{array} }\right\rbrace$	&$\add\left\lbrace {2,3,\begin{array}{cc}
			2 \\
			3
		\end{array} }\right\rbrace$  \\
	
			\hline
	
		\end{longtable}
		\end{center}
		(4) By corollary~\ref{3.10} (1), we have the following correspondence:		
\begin{center}		
	\begin{longtable}{|c|c|}
		\hline
		$	\mod B_{ice}$	&$\mod\Lambda_{ice}$  \\
		\hline\endfirsthead
		\hline
		$	\mod B_{ice}$	&$\mod\Lambda_{ice}$  \\
		\hline\endhead
		$\add\left\lbrace {0}\right\rbrace$&$\add\left\lbrace {4,1,\begin{array}{cc}
				4 \\
				1
		\end{array}}\right\rbrace$\\
		\hline
		$\add\left\lbrace {2}\right\rbrace$&$\add\left\lbrace {4,1,\begin{array}{cc}
				4 \\
				1
			\end{array},2,\begin{array}{cc}
			4 \\
			1 \\
			2
		\end{array},\begin{array}{cc}
	
		1 \\2
	\end{array}}\right\rbrace$  \\
		\hline
		$\add\left\lbrace {3}\right\rbrace$&$\add\left\lbrace {4,1,\begin{array}{cc}
				4 \\
				1
		\end{array},3}\right\rbrace$  \\
		
		\hline
		$\add\left\lbrace {\begin{array}{cc}
				2 \\
				3
		\end{array}}\right\rbrace$	&$\add\left\lbrace {4,1,\begin{array}{cc}
		4 \\
		1
		\end{array},\begin{array}{cc}
		2\\
		3
	\end{array},\begin{array}{cc}
	1 \\
	2\\3
\end{array},\begin{array}{cc}
4 \\
1 \\2\\3
\end{array} }\right\rbrace$  \\
		\hline
		$\add\left\lbrace {2,\begin{array}{cc}
				2 \\
				3
		\end{array} }\right\rbrace$	&$\add\left\lbrace {4,1,\begin{array}{cc}
		4 \\
		1
		\end{array},\begin{array}{cc}
		4 \\
		1 \\2
		\end{array},\begin{array}{cc}
		1\\
		2
		\end{array},2,\begin{array}{cc}
				2 \\
				3
		\end{array},\begin{array}{cc}
		4 \\
		1 \\2\\3
	\end{array},\begin{array}{cc}
	1 \\
	2 \\3
\end{array} }\right\rbrace$  \\
		\hline
		
		$\add\left\lbrace {2,3,\begin{array}{cc}
				2 \\
				3
		\end{array} }\right\rbrace$	&$\add\left\lbrace {4,1,\begin{array}{cc}
		4 \\
		1
		\end{array},\begin{array}{cc}
		4 \\
		1 \\2
		\end{array},\begin{array}{cc}
		1\\
		2
		\end{array},2,\begin{array}{cc}
		2 \\
		3
		\end{array},\begin{array}{cc}
		4 \\
		1 \\2\\3
		\end{array},\begin{array}{cc}
		1 \\
		2 \\3
	\end{array},3 }\right\rbrace$   \\	
		\hline	
	\end{longtable}
	\end{center}}	
		
		\end{example}
	
	\begin{example}\label{example-2}
		{\rm Let $\Lambda _{(\phi,\psi)}=\left( \begin{array} {cc}
	A	&M  \\
N		& B
	\end{array}\right) $ be a Morita ring as in \cite{ref8},
where $_{A}M_B$ is an $A$-$B$-bimodule, $_{B}N_A$ is a $B$-$A$-bimodule,
$\psi:M\otimes_{B}N\rightarrow{A}$ is a morphism of $A$-$A$-bimodules
and $\phi:N\otimes_{A}M\rightarrow{B}$ is a morphism of $B$-$B$-bimodules.
 By \cite[Proposition 2.4]{ref8},
there are two recollements
				$$\xymatrix@!=8pc{ \mod (A/\Im\psi) \ar[r] |{inc}& \mod {\Lambda_{\left( \phi,\psi\right) } }\ar@<-3ex>[l]
				\ar@<+3ex>[l] \ar[r]|{U_B} \ar[r]& \mod B
				\ar@<-3ex>[l]_{T_B} \ar@<+3ex>[l]^{H_B}} $$
and $$\xymatrix@!=8pc{ \mod (B/\Im\phi) \ar[r] |{inc}& \mod {\Lambda_{\left( \phi,\psi\right) } }\ar@<-3ex>[l]
				\ar@<+3ex>[l] \ar[r]|{U_A} \ar[r]& \mod A
				\ar@<-3ex>[l]_{T_A} \ar@<+3ex>[l]^{H_A}},$$
and then
Theorem~\ref{3.4}, ~\ref{3.10} (1), ~\ref{thm-bij}, ~\ref{thm-recoll},
~\ref{cor-tor} and ~\ref{4.3} apply. That is, we can construct ICE-closed subcategories, epibricks
and monobricks
of a Morita ring by its corner algebras. In particular, if $_BN$ (resp. $_AM$) is projective,
then the functor $H_B$ (resp. $H_A$) is exact and then Theorem~\ref{3.5} and Proposition~\ref{3.7} (2) apply.
If $M_B$ (resp. $N_A$) is projective,
then the functor $T_B$ (resp. $T_A$) is exact and then Theorem~\ref{3.5} apply.
			
		}
	\end{example}

	\bigskip
	
	\noindent {\footnotesize {\bf Acknowledgments} This work is supported by
		the National Natural Science Foundation of China (12061060),
		the project of Young and Middle-aged Academic and Technological leader of Yunnan
		(Grant No. 202305AC160005) and Yunnan Key Laboratory of Modern Analytical Mathematics and Applications (No. 202302AN360007).
		
	}

\end{document}